\documentclass[12pt,leqno,a4paper]{amsart}
\usepackage{amssymb,enumerate}

\newcommand{\CC}{{\mathbb{C}}}

\newcommand{\fA}{{\mathfrak{A}}}
\newcommand{\fS}{{\mathfrak{S}}}

\newcommand{\Char}{{\operatorname{char}\,}}
\newcommand{\Hom}{{\operatorname{Hom}}}

\newcommand{\PSL}{{\operatorname{L}}}
\newcommand{\SL}{{\operatorname{SL}}}
\newcommand{\GL}{{\operatorname{GL}}}
\newcommand{\PSU}{{\operatorname{U}}}
\newcommand{\SU}{{\operatorname{SU}}}
\newcommand{\Sp}{{\operatorname{Sp}}}
\newcommand{\PSp}{{\operatorname{S}}}
\newcommand{\OO}{{\operatorname{O}}}
\newcommand{\SO}{{\operatorname{SO}}}
\newcommand{\ch} {{\operatorname{ch}}}

\newcommand{\tw}[1]{{}^#1\!}

\newtheorem{thm}{Theorem}[section]
\newtheorem{lem}[thm]{Lemma}
\newtheorem{cor}[thm]{Corollary}
\newtheorem{prop}[thm]{Proposition}

\theoremstyle{definition}
\newtheorem{defn}[thm]{Definition}
\newtheorem{example}[thm]{Example}

\theoremstyle{remark}

\begin{document}

\title{Characteristic Polynomials and Fixed Spaces of Semisimple Elements}

\date{\today}

\author{Robert Guralnick}
\address{3620 S. Vermont Ave, Department of Mathematics, University of
Southern California, Los Angeles, CA 90089-2532, USA.}
\makeatletter
\email{guralnic@usc.edu}
\makeatother
\author{Gunter Malle}
\address{FB Mathematik, TU Kaiserslautern,
Postfach 3049, 67653 Kaisers\-lautern, Germany.}
\makeatletter
\email{malle@mathematik.uni-kl.de}
\makeatother

\dedicatory{Dedicated to Len Scott}
\thanks{The first author was partially supported by DMS 1001962. We thank
        Frank Calegari for his questions and comments.}

\begin{abstract}
Answering a question of Frank Calegari,  
we extend some of our earlier results on dimension of fixed point spaces
of elements in irreducible linear groups. We consider characteristic
polynomials rather than just fixed spaces. 
\end{abstract}

\maketitle


\section{Introduction} \label{sec:intro}

In \cite{GM1}, the authors answered a question of Peter Neumann and proved
that if $G$ is a nontrivial irreducible subgroup of $\GL_n(k)=\GL(V)$ with
$k$ a field, then there exists an element $g \in G$ with
$\dim C_V(g) \le (1/3) \dim V$  (where $C_V(g)$
denotes the fixed space of $g$ acting on $V$). The example $G=\SO_3(k)$ with
$k$ not of characteristic $2$ shows that $1/3$ is best possible.  

Frank Calegari asked if one could find $g \in G$ such that the characteristic
polynomial of $g$ acting on $V$ was of the form $(T-1)^e f(T)$ where
$f(1)\ne 0$ and $e < n/2 $. Calegari and Gee \cite{CG} are interested in the
irreducibility of the Galois representations associated to self-dual
cohomological automorphic forms for $\GL_n$, especially for small $n$. This
result is easily seen to be true for finite $G$ (or more generally compact $G$)
in characteristic $0$ by the orthogonality
relations (see \cite[Section 3]{GMa} for this simple proof). For finite groups
(or more generally for algebraic groups or if the characteristic is positive), 
the question reduces to finding a semisimple element $g \in G$ with
$\dim C_V(g)<n/2$
(recall an element in a linear group is semisimple if it is diagonalizable over
the algebraic closure --- for a finite group, this is equivalent
to saying that $\Char k$ does not divide the order of the element).

The authors in \cite[Thm.~1.3]{GM1} proved a conjecture of
Peter Neumann from 1966 about  the minimum  dimension of fixed spaces.
 In particular, if $G$ is finite and
the characteristic does not divide $|G|$, this gives:
 
\begin{thm}   \label{thm:gm1}
 Let $k$ be a field of characteristic $p \ge 0$ and $G$ a nontrivial finite
 subgroup of $\GL_n(k) = \GL(V)$ with $p$ not dividing $|G|$.
 If $V$ is an irreducible $kG$-module, then there exists a semisimple element
 $g \in G$ with $\dim C_V(g) \le (1/3) \dim V$.
\end{thm}

We note that this result depends upon the classification of finite
simple groups. However, the result had previously been unknown even
for solvable groups. 

The conclusion of the theorem no longer holds in non-coprime characteristic.
The simplest example is to take $G=A.C$ the semidirect product of an elementary
abelian group $A$ of order $8$ with the cyclic group of order $7$ acting
faithfully. Then $G$ has an irreducible
$7$-dimensional representation over any field of characteristic not $2$. In
particular, in characteristic $7$, we see that every nontrivial semisimple
element has a $3$-dimensional fixed space. More generally, if $p = 2^a-1$ is
a Mersenne prime, let $G=A.C$ with $A$ elementary abelian of order $p+1=2^a$
and $C$ of order $p$ acting faithfully on $A$. Then $G$ has an irreducible
representation of dimension $p$ in characteristic $p$  and every nontrivial
semisimple element has a fixed space of dimension $(p-1)/2$. Moreover,
taking direct products of this group with itself and the corresponding tensor
product of representations, one gets examples of arbitrarily large dimension.   

Even in characteristic $0$, the Example~6.5 in \cite{GM1} shows that one can
do no better than $(1/9) \dim V$ no matter how large the dimension of $V$.    
In this note, we will prove Calegari's inequality and show that one can do
better under various circumstances.

If we consider connected algebraic groups, we can obtain similar results that
do not depend upon the classification of finite simple groups. Since a compact real Lie group has only
semisimple elements and since the minimum value of $\dim C_V(g)$ is attained on
a nonempty open subset of $G$, we also obtain:

\begin{thm}   \label{thm:compact}
 Let $G\ne1$ be a connected compact real Lie subgroup of $\GL_n(\CC)=\GL(V)$.
 Assume that $C_V(G)=0$. Then the average dimension (with respect to Haar
 measure) of $C_V(g)$ is at most $(1/3) \dim V$.
\end{thm}

Let $\epsilon > 0$. It was shown in \cite[Thm. 6]{GLM} that if $G$ is compact
and connected
and $V$ is irreducible, then in fact the average dimension of $C_V(g)$ is less
than $\epsilon \dim V$ as long as $\dim V$ is sufficiently large.   

We have a similar result for Zariski dense subgroups of connected algebraic
groups.

\begin{thm}   \label{thm:algebraic}
 Let $G\ne1$ be a subgroup of $\GL_n(k)=\GL(V)$ with $k$ an algebraically
 closed field and $\Char k = p \ge 0$. Assume that $V$ is completely reducible,
 $C_V(G)=0$ and the Zariski closure of $G$ is connected. Then the set of
 semisimple $g \in G$ with $\dim C_V(g) \le (1/3) \dim V$ is open
 and Zariski dense in~$G$.
\end{thm}

An easy consequence of Theorem \ref{thm:algebraic} is:  

\begin{cor}   \label{cor:infinite}
 Let $G$ be an irreducible subgroup of $\GL_n(k)=\GL(V)$ with
 $\Char k = p \ge 0$. Assume that $G$ is infinite. Then there exists a
 semisimple $g \in G$ such that $\dim C_V(g)  \le (1/3) \dim V$.   
\end{cor} 

Thus, we are reduced to considering finite groups. 
For the general case, we can show:

\begin{thm}  \label{thm:general}
 Let $1\ne G \le \GL_n(k)=\GL(V)$ with $\Char k=p\ge 0$. Assume that $G$ acts
 irreducibly on $V$.
 \begin{enumerate}
  \item[\rm(a)] There exists a semisimple $g\in G$ with
   $\dim C_V(g)< (1/2)\dim V$.
  \item[\rm(b)] If $p > n+2$ or $p=0$, then there exists a semisimple $g\in G$ 
   with $\dim C_V(g) \le (1/3) \dim V$.
   \item[\rm(c)]  If $p$ does not divide $n$, then 
   there exists a semisimple $g \in G$ with
   $\dim C_V(g) \le (3/8) \dim V$.
  \item[{\rm(d)}]  If $n$ is prime and $2$ is a multiplicative generator
   modulo $n$, then there exists a semisimple $g \in G$ with
   $\dim C_V(g) \le (1/3) \dim V$.
   
 \end{enumerate}
\end{thm}

In particular, this shows that if $n \le 5$, there exists a semisimple
element $g\in G$ with $\dim C_V(g)  \le 1$. If $n$ is prime, one can prove
even stronger results:

\begin{thm} \label{thm:prime}
 Let $G$ be a finite irreducible subgroup of $\GL_n(k)$ with $n$ an odd prime.
 Assume that $\Char k = p  > 2n - 3$ (or $p=0$).
 There exists a semisimple element $x \in G$ with all eigenspaces of dimension
 at most $1$.
\end{thm}

The paper is organized as follows. In the next section, we deal with algebraic
groups and deduce the various results on algebraic and infinite groups.
We then consider various generation results about finite simple
groups. In particular, in Section~\ref{sec:gen} we prove the following results
that may be of independent interest:
 
\begin{thm}   \label{thm:generationresult}
 Let $G$ be a finite nonabelian simple group and $p$ be a prime.
 Then unless $(G,p)=(\fA_5,5)$, there exist $p'$-elements $x,y,z \in G$
 with $xyz=1$ such that $G = \langle x, y \rangle$.
\end{thm}
 
\begin{thm}   \label{thm:gen2}
 Let $G$ be a finite nonabelian simple group and $p$ be a prime.
 Then   there exist a pair of conjugate $p'$-elements that
 generate $G$.
\end{thm}
 
Note that an immediate consequence of Theorem \ref{thm:generationresult}
and Scott's Lemma \cite{scott} is:

\begin{cor}   \label{cor:simple}
 Let $G$ be a finite nonabelian simple subgroup of $\GL_n(k)=\GL(V)$. Assume
 that $C_V(G)=C_{V^*}(G)=0$.   If $G = \fA_5$ and $\Char k =5$, assume
 further that $V$ has no trivial composition factors. 
 Then there exists a semisimple element $g \in G$
 with $\dim C_V(g) \le  (1/3) \dim V$.
\end{cor} 

It is shown \cite[Thm.~6.1]{GM1} that in fact if $\epsilon>0$, $G$ is a
nonabelian finite simple group and $V$ is an irreducible $\CC G$-module,
then there exists $g \in G$ with $\dim C_V(g) < \epsilon \dim V$ as long as
$\dim V$ is sufficiently large (or equivalently $|G|$ is sufficiently large).   
This should be true for any algebraically closed field.  

In Section~\ref{sec:finite}, we consider finite groups and prove
Theorem~\ref{thm:general}(a)--(c). We then consider representations of prime
dimension, complete the proof of Theorem~\ref{thm:general} and prove
Theorem~\ref{thm:prime}. In the final
section, we give some examples relating to the divisibility of characteristic 
polynomials of representations (a question asked by Calegari \cite{CG}).  
 
\section{Infinite  Groups}  \label{sec:infinite} 
 
In this section we prove Theorems~\ref{thm:compact}, \ref{thm:algebraic}
and Corollary~\ref{cor:infinite}. These results do not require the
classification of finite simple groups.

We first prove Theorem \ref{thm:algebraic}. Let $k$ be an algebraically
closed field of characteristic $p \ge 0$. Let
$G$ be a subgroup of $\GL_n(k)=\GL(V)$. We assume that $G$ has no trivial
composition factors on $V$ and that $\Gamma$, the Zariski closure of $G$
is connected.

Our assumption is that $G$ (equivalently $\Gamma$) acts completely reducibly
on $V$. Thus, $\Gamma$ is reductive. Note that for any $e\ge0$,
$\{g \in \Gamma \mid \dim C_V(g) \le e\}$
is an open subvariety of $\Gamma$.  Moreover, the set of semisimple
elements of $\Gamma$ is also open.  
 
Let $S$ be a rational irreducible $k\Gamma$-module. If follows by
\cite[Thm.~3.3]{Gur98} that the set of pairs of semisimple elements
in $\Gamma$
which generate an irreducible subgroup on $S$ is a nontrivial open
subvariety of $\Gamma^2$.  Thus, the set of pairs of semisimple elements
in $\Gamma$ which have no fixed points on $V$ and whose product is semisimple
is also an open nonempty subvariety of $\Gamma^2$.  
Thus, this set intersects $G^2$ in a nonempty open subset of $G^2$.
Choose $x,y \in G$ such that $\langle x, y \rangle$ has no fixed points on $V$
with $x,y$ and $xy$ semisimple. 

By Scott's Lemma \cite{scott}, $\dim C_V(x)+\dim C_V(y)+\dim C_V(xy)\le\dim V$
and so some semisimple element $g\in G$ satisfies $\dim C_V(g)\le (1/3)\dim V$.
This shows that the set of semisimple elements of $\Gamma$ with
$\dim C_V(g) \le (1/3) \dim V$ is an open dense subvariety of $\Gamma$.
In particular, this set must intersect $G$ whence Theorem \ref{thm:algebraic}
holds. Theorem \ref{thm:compact} now follows immediately.  

We now prove Corollary \ref{cor:infinite}. Arguing as in \cite[Thm.~5.8]{GM1},
it suffices to work over an algebraically closed field. Let $\Gamma$ be the
Zariski closure of $G$ and $\Gamma^\circ$ the connected component of the
identity in $\Gamma$. Note that $\Gamma^\circ\ne1$ as $G$ is infinite.
Since $V$ is irreducible for $G$, hence for $\Gamma$, $\Gamma^{\circ}$ acts
completely reducibly without fixed points on $V$. By
Theorem~\ref{thm:algebraic}, the set of semisimple $g \in \Gamma^\circ$
with $\dim C_V(g) \le (1/3) \dim V$ contains a dense open subset of
$\Gamma^\circ$
and therefore intersects $G \cap \Gamma^\circ$ non-trivially, as required.

We close this section by showing that often one can do even better in
the case of algebraic groups. There is a version of the following theorem
for semisimple groups as well.

\begin{thm}   \label{thm:allss}
 Let $G$ be a simple simply connected algebraic group of rank $r$ at least $2$
 over an algebraically closed field $k$. Let $V$ be a completely reducible
 rational $kG$-module with $C_V(G)=0$. Let $g\in G$ be a regular semisimple
 element and assume that $g^3$ is not central (the latter can only fail if
 $G=\SL_3$). Then:
 \begin{enumerate}[\rm(a)]
  \item $\dim C_V(g) \le (1/3) \dim V$, and
  \item if $g^2$ is also regular, then every eigenspace of $g$ has dimension
   at most $(1/3) \dim V$.
 \end{enumerate}
\end{thm}

\begin{proof}
Let $C$ be the conjugacy class of $g$. We first prove (a). Let $X$ be the
variety of triples of elements all in $C$ with product $1$. By
\cite[Thms.~6.11, 6.15]{GM2}, this is an irreducible variety (of dimension
$2 \dim G - 3r$) and the set of triples in $X$ which generate
a subgroup $H$ such that each irreducible submodule of $V$ remains
irreducible for $H$ (and non-isomorphic irreducibles remain non-isomorphic)
is a dense open subvariety of $X$. Now (a) follows by Scott's Lemma.
 
If $g^2$ is regular semisimple, we consider the variety
$$Y = \{(x,y,z) \in C \times C \times C^{-2} \mid xyz=1\}.$$
Precisely as above, we see that the subset of $Y$ consisting of triples so
that the subgroup they generate has the same collection of irreducibles as $G$
is dense. Apply Scott's Lemma to the elements
$(\lambda^{-1} x, \lambda^{-1} y, \lambda^2 z)$ to conclude
that the $\lambda$-eigenspace of at least one of them has dimension at most
$(1/3) \dim V$.
\end{proof}

Note that we do not need to assume that $V$ is completely reducible in the
previous result.  The proof goes through verbatim as long as we assume
that $C_V(G)=C_{V^*}(G)=0$.

\section{Generation results} \label{sec:gen}

The purpose of this section is the proof of the following generation results:

\begin{thm}   \label{thm:3elt}
 Let $G$ be a finite non-abelian simple group, $p$ a prime. Then we have:
 \begin{enumerate}
  \item[\rm(a)] $G$ is generated by two conjugate $p'$-elements, and
  \item[\rm(b)] $G$ is generated by three $p'$-elements $x,y,z$ with $xyz=1$
   unless $(G,p)=(\fA_5,5)$.
 \end{enumerate}
\end{thm}

Note that $(\fA_5,5)$ is a true exception to the conclusion of
Theorem~\ref{thm:3elt}(b) since the largest element order of $\fA_5$ prime
to~5 is~3, and the triangle groups $G(l,m,n)$ with $l,m,n\le3$ are solvable.

In \cite[Thm.~1.1]{GM1} we showed that any finite non-abelian simple group
$G$ is generated by a certain triple of conjugate elements with product~1.
Thus the remaining task is to prove~(a) and~(b) for primes $p$ dividing the
common order of these elements. This will be shown in the subsequent
propositions. 

\begin{prop}
 Theorem~\ref{thm:3elt} holds for sporadic groups and the Tits group.
\end{prop}

\begin{proof}
For $G$ a sporadic simple group both parts follow from \cite[Table~4]{GM2},
where we exhibited a second generating systems $(x,y,z)$ for $G$ with
product~1, with $x\sim y$ and the orders of $x,y,z$ prime to those from
\cite{GM1}.
\end{proof}

\begin{prop}   \label{prop:alter}
 Theorem~\ref{thm:3elt} holds for alternating groups.
\end{prop}

\begin{proof}
For $G=\fA_n$ an alternating group, with $n\ge11$ odd, we produced in
\cite[Lemma~4.2]{GM1} a generating system $(x_1,y_1,z_1)$ of $n-2$-cycles
with product~1. In \cite[Cor.~17]{FG} there is given a generating system
$(x_2,y_2,z_2)$ with $x_2,y_2$ squares of an $n-1$-cycle and $o(z_2)=n$.
For $n\ge12$ even, we gave a generating triple of $n-3$-cycles
in \cite[Lemma 4.3]{GM1}, while \cite[Cor.~14]{FG} gives a generating
system with $x_2,y_2$ both $n-1$-cycles and $o(z_2)=n/2$. This gives the
claim unless $p=3$ divides $n$. In the latter case, by \cite[Cor.~2.2]{Be72}
for $n>12$ there exists two $n-5$-cycles with product of type $(n-2)(2)$,
and these generate a transitive group since otherwise, the three elements
would be in $\fS_{n-2}$ with the first two being in $\fA_{n-2}$. The $n-5$
cycle guarantees the group generated is primitive and \cite[Thm.~13.8]{W}
implies that they generate a $6$-transitive group. The result now follows
by \cite[Thm.~51.1]{W}.
The alternating group $\fA_{12}$ has a generating
triple consisting of elements of order~11.
\par
For $n=5,6,7,8,9,10$ we gave generating triples of orders 5,5,7,7,7,7
respectively in \cite[Lemma 4.4]{GM1}. For $n=6,7,8,9,10$ direct computation
shows that there also exist generating triples of orders 4,5,15,15,15. Note
that $\fA_5$ is generated by two 3-cycles (with product of order~5).
\end{proof}

\begin{prop}   \label{prop:exc}
 Theorem~\ref{thm:3elt} holds for exceptional simple groups of Lie type.
\end{prop}

\begin{proof}
For $G$ of exceptional Lie type different from
$\tw3D_4(q)$, we produced in \cite[Thm.~2.2]{GM2} a second generating system
consisting of conjugate elements, while for $\tw3D_4(q)$ in
\cite[Prop.~2.3]{GM2} we gave a generating triple containing two conjugate
elements. Moreover, in all cases, the element orders in these triples are
coprime to those from \cite{GM1}.
\end{proof}

Finally assume that $G$ is of classical Lie type. For $n\ge2$, we let 
$\Phi_n^*(q)$ denote the largest divisor of $q^n-1$ that is relatively
prime to $q^m-1$ for all $1\le m<n$.  

\begin{prop}   \label{prop:psl2}
 Theorem~\ref{thm:3elt} holds for $\PSL_2(q)$, $q\ge4$.
\end{prop}

\begin{proof}
The groups $\PSL_2(q)$ with $q\in\{4,5,9\}$ are isomorphic to alternating
groups, for which the claim follows from Proposition~\ref{prop:alter}. The
group $\PSL_2(7)$ has generating triples consisting of elements of order~7,
respectively of order~4. For $q\ge8$, $q\ne9$, we gave in
\cite[Lemmas~3.14, 3.15]{GM1}
generating triples for $\PSL_2(q)$ of orders $(q-1)/d$, where $d=\gcd(2,q-1)$.
Similarly, a direct calculation shows that there exist generating triples of
order $(q+1)/d$, which proves the claim.
\end{proof}

\begin{defn}
 Let's say that a pair $(C,D)$ of conjugacy classes of a group $G$ is
 \emph{generating} if $G=\langle x,y\rangle$ for all $(x,y)\in C\times D$.
\end{defn}

Note that by the result of Gow \cite{Gow}, if $(C,D)$ is a generating pair
for a finite group of Lie type consisting of classes of regular semisimple
elements, then we find generators $(x,y)\in C\times D$ with product in any
given (noncentral) semisimple conjugacy class, for example in $C^{-1}$.

\begin{prop}   \label{prop:Ln}
 Theorem~\ref{thm:3elt} holds for the groups $\PSL_n(q)$, $n\ge3$.
\end{prop}

\begin{proof}
Let $G=\PSL_n(q)$, $n\ge3$. Note that we may assume that $(n,q)\ne(3,2),(4,2)$
as $\PSL_3(2)\cong\PSL_2(7)$ and $\PSL_4(2)\cong\fA_8$ were already handled
above. In \cite[Prop.~3.13]{GM1} we showed that $G$ is generated by a triple
of elements of order $\Phi_n^*(q)$ when $n$ is odd, respectively of order
$\Phi_{n-1}^*(q)$ when $n$ is even. For $n\ne4$ and $(n,q)\ne(6,2)$, we showed
in \cite[Prop.~3.1 and~3.5]{GM2} that there also exist generating triples of
elements of orders coprime to the former ones, of the type stated in
Theorem~\ref{thm:3elt}. The group $\PSL_6(2)$ is generated by a triple of
elements of order~7.
\par
Now consider $G=\SL_4(q)$. Let $C$ be a conjugacy class of regular semisimple
elements of order $(q^4-1)/(q-1)$. According to \cite[Lemma~2.3]{GM1}, the only
maximal subgroups of $\SL_4(q)$ which might contain an element $x\in C$ are
the normalizer of $\Omega_4^-(q)$, of $\Sp_4(q)$ or of $\GL_2(q^2)\cap\SL_4(q)$.
But in the first two of these groups the centralizer order of elements of
order $q^2+1$ is not divisible by $(q^2+1)(q+1)$ when $q>3$. Thus, any $x\in C$
lies in a unique maximal subgroup of $G$. Let $C_2$ denote a class of regular
semisimple elements in a maximal torus of order $(q^2-1)(q-1)$ of $\SL_4(q)$.
Then this does not intersect the normalizer of $\GL_2(q^2)\cap\SL_4(q)$, so
$(C_1,C_2)$ is a generating pair. By \cite{Gow} there exist
pairs with product in $C_1^{-1}$, which must generate. Now pass to the quotient
of $\SL_4(q)$ by its center. The group $\PSL_4(3)$ has a generating triple
with elements of order~5.
\end{proof}

\begin{prop}   \label{prop:Un}
 Theorem~\ref{thm:3elt} holds for the unitary groups $\PSU_n(q)$, $n\ge3$,
 $(n,q)\ne(4,2)$.
\end{prop}

\begin{proof}
Let $G=\PSU_n(q)$, $n\ge3$. The case $n=3$ was already treated in
\cite[Prop.~3.1]{GM2}. In \cite[Prop.~3.11 and~3.12]{GM1} we showed that $G$ is
generated by a triple of elements of order $\Phi_{2n}^*(q)$ when $n$ is odd,
respectively of order $\Phi_{2n-2}^*(q)$ when $n$ is even. For $n\ge8$ we
showed in \cite[Prop.~3.6]{GM2} that there also exist
generating triples of elements of orders coprime to the former ones, of the
form required in Theorem~\ref{thm:3elt}. For the remaining $n$ we argue in
$G=\SU_n(q)$ and then pass to the quotient by the center. For $n=7$ let $C_1$
contain elements of type $6^+$ and $C_2$ elements of type $5^-\oplus2^+$,
for $n=5$ let $C_1$ contain elements of type $4^+$ and $C_2$ elements of
type $3^-\oplus2^+$. Then any pair from $C_1\times C_2$ generates an
irreducible subgroup, and by \cite[Thm.~2.2]{GM1} that can't be proper, when
$(n,q)\ne(5,2)$. We conclude using \cite{Gow}. The group $\PSU_5(2)$ has
generating triples of elements of order~5. 
\par
If $n=4$ or $6$, the claim of Theorem~\ref{thm:3elt}(b) holds by
\cite[Prop.~3.6]{GM2}. For Theorem~\ref{thm:3elt}(a), the group $\PSU_6(2)$
has a generating triple of elements of order~7. Else let's take $C_1$ a class
of elements of order $\Phi_{2n-2}^*(q)$, $C_2$ a class of elements of order
$\Phi_4^*(q)$ when $n=4$, $\Phi_3^*(q)\Phi_6^*(q)$ when $n=6$, and $C_3$ one
of $C_1,C_2$. Then the arguments in the proof of \cite[Prop.~3.12]{GM1} go
through, with even better estimates, since the number of characters not
vanishing on both $C_1$ and $C_2$ is smaller, and the same for the number of
maximal subgroups containing elements from both classes. It follows that there
exist generating triples with product~1 in
$C_1\times C_1\times C_2$ and in $C_1\times C_2\times C_2$. Since the orders
in the two classes are relatively coprime, this gives the result. 
\end{proof}

\begin{prop}   \label{prop:O2n+1}
 Theorem~\ref{thm:3elt} holds for the orthogonal groups $\OO_{2n+1}(q)$,
 $n\ge3$.
\end{prop}

\begin{proof}
The group $G=\OO_{2n+1}(q)$ possesses a generating triple of elements of
order $\Phi_{2n}^*(q)$, by \cite[Prop.~3.7 and~3.8]{GM1}. For $n\ge7$, we
produced in \cite[Prop.~7.9]{GM2} a generating pair $(C,D)$ of conjugacy
classes containing regular semisimple elements of orders prime to
$\Phi_{2n}^*(q)$. For $n=4,5,6$, let $C$ contain elements of type
$(n-1)^-\oplus1^-$, $D$ elements of type $(n-2)^-\oplus2^+$. Then the group
generated by $(x,y)\in C\times D$ acts irreducibly or is contained in a
$2n$-dimensional orthogonal group. By consideration of suitable Zsigmondy
primes, the latter can possibly only occur when
$q\le3$, which we exclude for the moment. Otherwise, an application of
\cite[Cor.~3.4]{GM2} shows that $(C,D)$ is a generating pair, and we conclude
using \cite{Gow}. The groups $\OO_9(2)\cong\PSp_8(2)$, $\OO_9(3)$,
$\OO_{11}(2)\cong\PSp_{10}(2)$, $\OO_{11}(3)$, $\OO_{13}(2)\cong\PSp_{12}(2)$,
$\OO_{13}(3)$, possess generating triples with
elements of orders 7, 13, 31, 41, 31, 61 respectively, 
Finally, for $n=3$, we argued in \cite[Prop.~3.8]{GM2} that conjugacy
classes of regular semisimple elements of types $3^+$, $2^-\oplus1^+$ form a
generating pair in $\OO_7(q)$ for $q\ge5$, and we produced generating triples
of orders 7, 13 and~17 for $\OO_7(2)\cong\PSp_6(2)$, $\OO_7(3)$, $\OO_7(4)$
respectively.
\end{proof}

\begin{prop}   \label{prop:Sp2n}
 Theorem~\ref{thm:3elt} holds for the symplectic groups $\PSp_{2n}(q)$,
 $n\ge2$, $(n,q)\ne(2,2)$.
\end{prop}

\begin{proof}
Note that we may assume that $q$ is odd when $n\ge3$ by the result of
Proposition~\ref{prop:O2n+1}. The group $G=\PSp_{2n}(q)$ possesses a
generating triple of elements of order $\Phi_{2n}^*(q)$, resp. order~5 when
$(n,q)=(2,3)$, by \cite[Prop.~3.8]{GM1}. For $n\ge3$, $(n,q)\ne(4,3)$, we
found in \cite[Prop.~7.8]{GM2} a generating pair $(C,D)$ of conjugacy classes
containing regular semisimple elements of orders prime to $\Phi_{2n}^*(q)$,
and we may conclude as usual. The group $\PSp_8(3)$ has a generating triple
with elements of order~13. \par
For $\PSp_4(q)$ with $q\ge3$, the claim was already proved in
\cite[Prop.~3.1]{GM2}.
\end{proof}

\begin{prop}   \label{prop:O2n+}
 Theorem~\ref{thm:3elt} holds for the orthogonal groups $\OO_{2n}^+(q)$,
 $n\ge4$.
\end{prop}

\begin{proof}
Let $G=\OO_{2n}^+(q)$, $n\ge4$. In \cite[Prop.~3.10]{GM1} we showed that $G$
is generated by a triple of elements of order dividing $\Phi_{2n-2}^*(q)(q+1)$.
For $n\ne4$ we showed in \cite[Prop.~3.10]{GM2} that there also exist
generating pairs of conjugate regular semisimple elements of orders coprime
to the former ones, and we may conclude as before.
\par
For the $n=4$ we may assume that $q\ge3$, since for $\OO_8^+(2)$ there exist
generating triples of elements of order~9, and also of elements of order~7, by
\cite[Prop.~3.10]{GM1} and \cite[Prop.~3.10]{GM2}. Let $C_1$ contain regular
semisimple elements of order $(q^2+1)/d$, with $d=\gcd(2,q-1)$, in a maximal
torus $T$ of order $(q^2+1)^2/d^2$, and $C_2$ the image of $C_1$ under triality.
Let $(x,y)\in C_1\times C_2$. Then $H:=\langle x,y\rangle$ contains $T$ up to
conjugation. By \cite[Table~I]{Kl87} the only maximal subgroup of
$\OO_8^+(q)$ with this property is $M=(\OO_4^-(q)\times\OO_4^-(q)).2^2$.
According to \cite{Gow} there are $(x,y)$ with product of order a Zsigmondy
prime divisor of $q^2+q+1$, and hence not contained in $M$. This gives a
triple as in Theorem~\ref{thm:3elt}(b).
\par
We now show Theorem~\ref{thm:3elt}(a) for $n=4$ and $q \ge 3$.  Let $C_1$ be
a conjugacy class of elements of order $\Phi_6^*(q)$. Let $C_2$
be a conjugacy class of regular semisimple elements of order 
$(q^4-1)/4$ (or $q^4-1$ if $q$ is even) having precisely two 
non-trivial invariant subspaces.   Let $C_3$ be either $C_1$ or $C_2$.
Arguing as in \cite[Prop.~3.10]{GM1}, one gets a lower bound for the number of
triples $(x,y,z) \in C_1 \times C_2 \times C_3$ with product $1$ (the bound
is actually much better than in \cite{GM1} since there will be many fewer
characters not vanishing on $C_1$ and $C_2$). Similarly, one gets an
upper bound for the number of non-generating such triples (again the bound
is much better than that given in \cite{GM1} since there are many fewer maximal
subgroups --- for example, it is clear the group generated is irreducible).  
It follows that two elements of $C_1$ or $C_2$ will generate.  Since
the orders of the elements in $C_1$ and $C_2$ are relatively prime,
the result follows.  
\end{proof}

\begin{prop}   \label{prop:O2n-}
 Theorem~\ref{thm:3elt} holds for the groups orthogonal $\OO_{2n}^-(q)$,
 $n\ge4$.
\end{prop}

\begin{proof}
The group $G=\OO_{2n}^-(q)$ possesses a generating triple of elements of
order $\Phi_{2n}^*(q)$, by \cite[Prop.~3.6]{GM1}. For 
$(n,q)\notin\{(4,2),(4,4),(5,2),(6,2)\}$, we produced in \cite[Prop.~7.10]{GM2}
a generating pair $(C,D)$ of conjugacy classes of $G$ containing regular
semisimple elements of orders prime to $\Phi_{2n}^*(q)$.
Now by \cite{Gow} there exist triples in $C\times C\times D$, for example.
Direct computation shows that the groups $\OO_8^-(2)$, $\OO_8^-(4)$,
$\OO_{10}^-(2)$, $\OO_{12}^-(2)$ possess generating triples with elements of
orders 7, 13, 17, 31 respectively, which are again prime to $\Phi_{2n}^*(q)$.
\end{proof}

\section{Finite Groups}  \label{sec:finite} 

Here, we prove Theorem~\ref{thm:general}(a)--(c). By
Corollary~\ref{cor:infinite} it suffices to consider finite groups.

Fix a field $k$ of characteristic $\Char k = p$. By Theorem~\ref{thm:gm1},
we may assume that $p > 0$.  Assume that
$1\ne G \le \GL_n(k)=\GL(V)$ is irreducible on $V$. By extending scalars, we
can reduce to the case that $k$ is algebraically closed ($V$ would at worst be
a direct sum of Galois conjugates of a given irreducible module).

Let $N$ be a minimal normal subgroup of $G$. Thus, $N$ acts completely
reducibly on $V$ and without fixed points. We break up the argument depending
upon the structure of $N$.

\begin{lem}   \label{lem:odd}
 If $|N|$ is odd, then $N$ is an $r$-group for some odd prime $r \ne p$
 and there exists a $p'$-element $g \in N$ with
 $\dim C_V(g) <  (1/r) \dim V \le (1/3) \dim V$. 
\end{lem}

\begin{proof}
Since $|N|$ is odd, $N$ is solvable and since it is a minimal normal
subgroup, it must be an $r$-group with $r \ne p$. Now apply
\cite[Cor.~1.3]{GMa}. 
\end{proof} 

\begin{lem}  \label{lem:even}
 If $N$ is an elementary abelian $2$-group, then $p \ne 2$ and 
 \begin{enumerate}[\rm(a)]
  \item there exists an involution $g \in N$ with $\dim C_V(g)<(1/2)\dim V$;
  \item if $p$ does not divide $\dim V$, then there exists a $p'$-element
   $g \in G$ with $\dim C_V(g)  \le (1/3) \dim V$.
 \end{enumerate}
\end{lem}

\begin{proof}
Clearly, $p \ne 2$. By \cite[Cor.~1.3]{GMa}, there exists $g \in N$
with $\dim C_V(g) < (1/2) \dim V$.  

So we may assume that $p$ does not divide $\dim V$.
If $N$ is central, then any $1 \ne g \in N$ satisfies $C_V(g)=0$ and there
is nothing to prove. Otherwise, $V = \bigoplus V_i$ where the $V_i$ are the
distinct $N$-eigenspaces. Since $V$ is irreducible, $G$ permutes the $V_i$
transitively. By \cite[Thm.~1]{FKS}, there exists $x \in G$ 
of prime power order $r^a$ having no fixed points on the set of $V_i$.
Since $p$ does not divide $\dim V$, $r \ne p$. If $r$ is odd, then every
orbit of $x$ has size at least $3$ whence $\dim C_V(x) \le (1/3) \dim V$. So
we may assume that $r =2$. It follows as in the proof of \cite[Thm.~5.8]{GM1}
that the average dimension of the fixed point spaces of elements in the
coset $xN$ is at most $(1/4) \dim V$.  Since every element in $xN$ is a
$2$-element, the result follows. 
\end{proof}

The remaining case is when $N$ is a direct product of $t \ge 1$ isomorphic
copies of a nonabelian simple group $L$. Let $W$ be an irreducible
$kN$-submodule of $V$. By reordering, we may write
$W = U_1 \otimes\cdots\otimes U_t$
where each $U_i$ is an irreducible $kL$-module with $U_i \cong k$ if and only
if $i > m$ for some $m > 1$. First we note that the proof of
\cite[Cor.~5.7]{GM1} gives:  

\begin{lem}   \label{lem:natural}
 Let $L$ be a simple group of Lie type over a finite field of
 characteristic~$p$ and let $E = L \times\cdots\times L$. Let $k$ be an
 algebraically closed field of characteristic~$p$, $V$ a completely reducible
 $kE$-module with $C_V(E)=0$. Then there exists a semisimple element
 $x \in E$ with $\dim C_V(x)  \le (1/3) \dim V$.
\end{lem}

Actually, the proof in \cite{GM1} does not work if $L \cong \PSL_2(5)$
with $p=5$, or $L\cong\PSL_2(7)$ with $p=7$.      
A more complicated proof can be given in these cases to show
that the result is still true.   The result also follows easily if 
$L \ne \PSL_2(q)$  using  Theorem \ref{thm:allss}.

\begin{lem}   \label{lem:3dim}  
 Let $L$ be a nonabelian finite simple group and let $E=L\times\cdots\times L$
 ($t$ copies). Let $k$ be an algebraically closed field of characteristic $p$,  
 $V=U_1 \otimes\cdots\otimes U_t$ a nontrivial irreducible $kE$-module with
 $\dim U_i=3$ for some $i$.  
 There exists a $p'$-element $x \in E$ (independent of $V$)
 such that all eigenspaces of $x$ have dimension at most $(1/3) \dim V$.
\end{lem}

\begin{proof}
First assume that $t=1$, so $\dim V=3$. By inspection, we can choose $x$ of
odd prime order with distinct eigenvalues (indeed unless $p=3$, we can take
$x$ of order $3$).

Suppose that $t > 1$. Let $y$ be the element of $E$ with all coordinates
equal to the $x$ chosen above. Since $x$ has all eigenspaces of dimension
at most $(1/3) \dim U_i$ on $U_i$, the same is true on $V$.
\end{proof}
 
\begin{lem}  \label{lem:last}
 Let $L$ be a finite group with $L = \langle x, y \rangle$ and let  
 $E= L \times\cdots\times L$. Let $k$ be an algebraically closed field
 of characteristic $p$, $V= U_1 \otimes\cdots\otimes U_t$ a nontrivial
 irreducible $kE$-module with $\dim U_i \ge 4$ for some $U_i$.  
 Let $G$ be a diagonal copy of $L$ in $E$.  Let $x_1, y_1$ and $z_1$ be
 elements of $G$ with each coordinate $x,y$ or $z=(xy)^{-1}$ respectively. Then
 $\dim C_V(x_1) + \dim C_V(y_1) + \dim C_V(z_1) \le (9/8) \dim V$.
 In particular, 
 $$\min\{\dim C_V(x_1),  \dim C_V(y_1),  \dim C_V(z_1) \} \le (3/8) \dim V.$$
\end{lem}

\begin{proof}
The assumption that some $U_i$ has dimension at least~$4$ gives that
$\dim C_V(G)  \le (1/16) \dim V$. Indeed, assume that $\dim U_1\ge4$. 
Then $C_V(G) \cong \Hom_G(U_1^*, U_2 \otimes\cdots\otimes U_t)$.
Since $\dim U_1\ge4$,
$$\dim C_V(G)\le(1/4)\dim (U_2\otimes\cdots\otimes U_t)
  \le(1/16)\dim V.$$
Similarly, $\dim C_{V^*}(G) \le (1/16) \dim V$. By Scott's Lemma \cite{scott}
$$\begin{aligned}
  \dim C_V(x_1)+&\dim C_V(y_1)+\dim C_V(z_1)\\
  &\le\dim V+\dim C_V(G)+\dim C_{V^*}(G)\le (9/8) \dim V.
\end{aligned}$$
The result follows. 
\end{proof} 

\begin{cor}   \label{cor:last}
 Let $G$ be a finite group, $k$ an algebraically closed field of
 characteristic $p$ and $V$ a faithful irreducible $kG$-module.
 Let $E= L\times\cdots\times L$ be a minimal normal subgroup of $G$ with $L$
 a nonabelian simple group. Then there exists
 a $p'$-element $x \in E$ with $\dim C_V(x) \le (3/8) \dim V$.
\end{cor}

\begin{proof}
If $L$ is of Lie type in characteristic $p$, the result follows by
Lemma~\ref{lem:natural}. So assume that this is not the case.  

Let $W$ be any irreducible $kE$-submodule of $V$.  Then
$W=U_1 \otimes\cdots\otimes U_t$ where each $U_i$ is an irreducible
$kL$-module.

If $\dim U_i=2$, then $p=2$ and $L\cong\SL_2(2^f)$ is of Lie type in
characteristic $2$, whence the result by Lemma~\ref{lem:natural}.

If $\dim U_i=3$ for some $i$, then this will be the case for every
irreducible $kE$-submodule (since any such is a twist of $W$) and
Lemma \ref{lem:3dim} applies. 

In the remaining case, $\dim U_i > 3$ for every nontrivial $U_i$
(and similarly for every twist of $W$). 
By Theorem \ref{thm:generationresult}, we can choose $p'$-elements $x,y,z\in L$
which generate $L$ and have product $1$ aside from $(L,p)=(\fA_5,5)$.
Note that $\fA_5\cong\PSL_2(5)$, so the latter is of Lie type in defining
characteristic, a case we already dealt with (alternatively, it would follow
that each nontrivial $U_i$ has dimension $5$ and so if $x$ is an element
of $E$ with each coordinate of order $3$,
$\dim C_W(x) \le (9/25) \dim V < (3/8) \dim W$). 
It follows by Lemma \ref{lem:last} that there exists a $p'$-element
$g \in E$ with $\dim C_V(g) \le (3/8) \dim V$. 
\end{proof} 

We can now prove the first three parts of Theorem \ref{thm:general}.  

\begin{proof}[Proof of Theorem~\ref{thm:general}]
Parts (a) and (c) follow from Lemmas~\ref{lem:odd}, \ref{lem:even} and
Corollary~\ref{cor:last}.   

Now assume
that $p > \dim V + 2$. By Lemma~\ref{lem:odd}, we may assume that $G$ has no
odd order non-trivial normal subgroups. If $O_2(G) \ne 1$, Lemma~\ref{lem:even}
applies since $p$ does not divide $\dim V$. So we may assume that $F(G)=1$.
Let $N=L \times\cdots\times L$ be a minimal normal subgroup of $G$ with $L$
a nonabelian simple group.  By \cite[Thm.~B]{Gur99}, it follows that one of the
following holds:
 \begin{enumerate}
 \item $p$ does not divide $|N|$; 
 \item $L$ is a finite group of Lie type in characteristic $p$; or
 \item $p=11$,  $N=J_1$ and $\dim V=7$.
\end{enumerate}
Thus by \cite[Cor.~5.7]{GM1} there exists $x \in N$
with $\dim C_V(x) \le (1/3) \dim V$.  If $x$ is a $p'$-element, we are 
done.  So we may assume that $p$ divides the order of $x$ (and so $|N|$).

If $N=J_1$, $p=11$ and $\dim V=7$, an element $g$ of order $19$ satisfies
$\dim C_V(g)=1 < (1/3)\dim V$. If $L$ has Lie type, then
Lemma~\ref{lem:natural} yields a semisimple element $g \in N$ with
$\dim C_V(g) \le (1/3) \dim V$. This completes the proof of (b). 
\end{proof}

\section{Prime Degree}

Recall that a group is called quasi-simple if $G$ is perfect and $G/Z(G)$
is a nonabelian simple group.   

Let $n$ be an odd prime. Let $k$ be a field with $\Char k = p$.   
Let $G \le \GL_n(k)= \GL(V)$ be irreducible and finite. If $G$ is not
absolutely irreducible, then $G$ must be cyclic and any generator has
distinct eigenvalues on $V$ (and is semisimple). So assume that $G$ is
absolutely irreducible and $k$ is algebraically closed.

\begin{lem}  \label{lem:syl}
 If $p \ne n$ and the Sylow $n$-subgroup is not abelian, then $G$ contains a
 subgroup $A$ of order $n$  such that $V$ is a free $kA$-module.
\end{lem}

\begin{proof}
Let $N$ be a minimal nonabelian $n$-subgroup of $G$. Then $N$ acts irreducibly
and is extraspecial, whence the result is clear.
\end{proof}  

\begin{lem} \label{lem:prime-imp}
 Suppose that $G$ acts imprimitively.
 \begin{enumerate}
  \item[\rm(a)] If $n=p$, then the Sylow $n$-subgroup $S$ of $G$ has order~$n$
   and $V$ is a free $kS$-module.
  \item[\rm(b)] If $n \ne p$, then there exists an element $x \in S$ with
   order a power of $n$ having all eigenspaces of dimension at most $1$. 
 \end{enumerate}
\end{lem}

\begin{proof}
Since $n$ is prime, $G$ imprimitive implies that $G$ permutes $n$ one
dimensional (linearly independent) subspaces. Thus, $G$ surjects onto a
transitive permutation group of degree $n$.  In particular, $n$ divides the
order of $G$. Let $N$ be the normal subgroup of $G$ stabilizing each of
the $n$ one dimensional spaces. If $p = n$, then $N$ has order prime to $p$,
whence $S$ has order $n$ and the result follows.

So assume that $p \ne n$.  Let $x \in G$ be an $n$-element with $x$ not in $N$.
Then $x^n$ is central in $\GL(V)$ and its minimal polynomial is $x^n-a$ for
some $a \in k^\times$, whence the result. 
\end{proof}

Note that since $n$ is prime if $N$ is a minimal normal noncentral
subgroup of $G$, then either $N$ is an elementary abelian
$r$-group for some prime $r \ne p$ or $N$ acts irreducibly.   

If $N$ acts irreducibly, then either $N$ is an $n$-group and 
$p \ne n$, whence $N$ contains a (semisimple) element with distinct
eigenvalues or $N$ is quasi-simple.  If $Z(N) \ne 1$, then $N$ contains
a semsimple element $x$ with $C_V(x)=0$.  If $N$ is simple, then 
Corollary \ref{cor:simple} implies that there exists $x \in N$ semisimple
with $\dim C_V(x)  \le (1/3)\dim V$ (if $N=\fA_5$, the result follows by
inspection).  

If $N$ is abelian and not a $2$-group, then  \cite[Cor.~ 1.3]{GMa} implies
that there exists a (semisimple) $x \in N$ with  $\dim C_V(x) < (1/3) \dim V$.
Finally, if $N$ is a $2$-group and $2$ is a multiplicative generator
modulo $n$, then $|N|=2^{n-1}$ (because $G$ permutes transitively the $n$
eigenspaces of $N$ and the smallest irreducible module of a cyclic group of
order $n$ in characteristic $2$ has size $2^{n-1}$). It follows that there
exists $x \in N$ with $-x$
a reflection, whence $\dim C_V(x) = 1 \le (1/3) \dim V$.   

In particular, we have proved part (d) of Theorem \ref{thm:general}.  
 
We next consider quasi-simple groups and  first show:

\begin{thm} \label{thm:primedim1}
 Let $G$ be a finite quasi-simple group of Lie type in characteristic $p$.
 Let $k$ be an algebraically closed field of characteristic $p$. Suppose that
 $V$ is a faithful irreducible $kG$-module of odd  prime dimension $n \le p$.  
 Then $V$ is a twist of a restricted module and one of the following holds:
 \begin{enumerate}[\rm(1)]
  \item  $G=\SL_2(q)$; 
  \item  $G=G_2(q)$ or ${^2}G_2(q)'$, $n=7$;
  \item  $G=\Omega_n(q)$ and $V$ is a Frobenius twist of the natural module; or 
  \item  $G=\SL_n(q)$ or $\SU_n(q)$ and $V$ is a Frobenius twist of the
   natural module or its dual.
 \end{enumerate}
 Moreover, either there exists a semisimple element $x \in G$ with  all
 eigenspaces of dimension at most $1$ on $V$ or $G=\SL_2(p)$ and $p \le 2n - 3$.
 In all cases, there exists a semisimple element $x \in G$ with $\dim C_V(x) \le 1$. 
\end{thm}

\begin{proof}  For the first part, 
it suffices to prove the result for algebraic groups. Let $V=L(\lambda)$ where
$\lambda$ is the highest weight for $V$. By the Steinberg tensor product
theorem and the fact that $n$ is prime, $V$ is a Frobenius twist of some
restricted irreducible. So we may assume that $\lambda$ is $p$-restricted.
By \cite{jantzen}, it follows that $L(\lambda)$ is also the Weyl module
whence the Weyl dimension formula holds for $L(\lambda)$. Thus, it suffices
to work in characteristic $0$. The result in that case is due to Gabber,
see \cite[1.6]{katz}.

Aside from the first case, any regular semisimple element of sufficiently large order
will have $n$ distinct eigenvalues on $V$.  Suppose that $G=\SL_2(q)$.
Let $x \in G$ have order $q+1$. If two distinct weights for a restricted
module can coincide on $x$, then $2 \dim V - 2 \ge q + 1$. This can only
occur if $q=p \le 2n -3$. Moreover, no nontrivial weight vanishes on an element
of order $q + 1$. Since a restricted irreducible $k \SL_2(q)$-module has
distinct weights, the result follows. 
\end{proof}
 
\begin{thm}  \label{thm:primedim2}
 Let $n$ be an odd prime. Let $G\le\GL(V)=\GL_n(k)$ be a finite irreducible
 quasi-simple group, where $k$ is an algebraically closed field of
 characteristic $p$.  
 \begin{enumerate}
  \item[\rm(a)]  If $p > \max\{2n - 3, n + 2\}$ or $p=0$, then there  
   exists a semisimple $x \in G$ with all eigenvalues distinct.
  \item[\rm(b)] If $p > n + 2$, there exists a semisimple element $x \in G$
   with $\dim C_V(x) \le 1$.
 \end{enumerate}
\end{thm}
 
\begin{proof}
If $p$ divides $|G|$, then by \cite[Thm.~B]{Gur99}, $G$ is a finite group of
Lie type in characteristic $p$ (or $G=J_1, p=11$ and $n=7$, where we may take
$x$ of order~19) and Theorem~\ref{thm:primedim1} applies.

If $p$ does not divide $|G|$, then either $p=0$ or $V$ is the reduction of a
characteristic $0$ module. The list of possible groups and modules is given
in \cite[Thm.~1.2]{DZ}.   
It is straightforward to see that the conclusion holds for these groups
(most of the examples are related to Weil representations).  
\end{proof} 

Theorem \ref{thm:prime} now follows from the previous results aside from the
case $n=3$ and $\Char k =5$. In that case, any noncentral semisimple element
of order greater than $2$ has distinct eigenvalues.
The next example shows that we do need some restriction on the characteristic.

\begin{example}
Let $k$ be an algebraically closed field of positive characteristic $p$.
Let $G= \SL_p(k)= \SL(V)$.   Then $G$ acts by conjugation on 
$W:=\mathrm{End}(V)$.  Since every semisimple $g\in G$ is centralized
by a maximal torus, we see that $\dim C_W(g) \ge p$.  Note that $W$
is a uniserial module with two trivial composition factors and an irreducible
composition factor  $V$ of dimension $p^2 -2$.  Clearly,  $\dim C_V(g) \ge
\dim C_W(g) - 2 \ge p-2$ for any semisimple element $g$ of $G$ (and
since semisimple elements are Zariski dense, this is true for any $g \in G$).
Note that $\dim V$ can be prime (eg, this is true for $p=5, 7, 13$).
The same holds for $G(p^a):=\SL_p(p^a)$ for any $a \ge 1$.
\end{example}

\section{Characteristic Polynomials of Representations}

Let $G$ be a group and $V$ a finite dimensional $kG$-module with 
$k$ a field of characteristic $p \ge 0$. Let $\ch_V$ denote the function
from $G$ to $k[x]$ defined by $\ch_V(g) = \det (xI - g)$. Note that two modules
have the same function if and only if their composition factors are the same.
Our results on bounds for $\dim C_V(g)$ for some $g \in G$ can be phrased in
asking: given a $kG$-module $V$ what is the largest power of $\ch_k$ that
divides $\ch_V$?

Frank Calegari asked what one could say if $V$ and $W$ are two irreducible
$kG$-modules and $\ch_V$ divides $\ch_W$.   Calegari and Gee \cite{CG} used
this information to study Galois representations in very small dimensions. 

While we suspect that this does impose some constraints on the representations, 
we give some examples to show that it is not that rare (at least for groups
of Lie type and algebraic groups in the natural characteristic).

\begin{example} \label{tensor}
 Let $G$ be a simple algebraic group over an algebraically closed field of
 characteristic $p > 0$.  Let $V$ be an irreducible $kG$-module. 
 By Steinberg's tensor product theorem, $V = V_0 \otimes\cdots\otimes V_m$
 where $V_i$ is a twist of a restricted module by the $i$th power of Frobenius.
 If $0$ is a weight for some $V_j$, then clearly $\ch_V$ is a multiple of
 $\ch_{V_j'}$, where $V_j'$ is the tensor product of all the $V_i, i \ne j$.
\end{example}

The following example was shown to us by N. Wallach (in particular see
\cite{Wa}).

\begin{example} \label{cartan mult}
 Let $k$ be an algebraically closed field of characteristic~$0$. Let $G$ be
 a simple algebraic group over $k$. Let $\lambda$ and $\mu$ be dominant
 weights with $\mu$ in the root lattice.  Then
 $\ch_{V(\lambda)}$ divides $\ch_{V(\lambda + \mu)}$.
\end{example}

It follows the same is true in positive characteristic $p$ as long as $p$ is sufficiently
large (depending upon $\lambda$ and $\mu$).   Here are a few cases where one
can compute this directly.   We give one such case.

\begin{example}  \label{natural1}  Let $k$ be an algebraically closed field of characteristic
$p \ge 0$.  Let $G=\SL_n(k)$ and let $V = V(\lambda_1)$
be the natural module.     Then $\ch_{V((s+n) \lambda_1)}$ is a multiple of 
$\ch_{V(s \lambda_1)}$ for $p > s + n$ (or $p=0$).
\end{example}

\vskip 3pc

\end{document}